\documentclass[12 pt]{amsart}

\usepackage{latexsym, amsmath, amssymb, longtable, booktabs,amscd}
\usepackage[square]{natbib}
\usepackage[english]{babel}
\usepackage[utf8x]{inputenc}
\usepackage{hyperref}
\usepackage{comment}
\bibpunct{[}{]}{,}{n}{}{;}
\usepackage{etoolbox}

\usepackage[colorinlistoftodos]{todonotes}
\usepackage{graphicx}
\usepackage{cleveref}
\usepackage{amsthm}
\numberwithin{equation}{section}

\usepackage{tikz-cd}

\usepackage{stackengine}

\newcommand\N{\mathbb{N}}

\newcommand\I{{\mathbb{I}}}

\newcommand\cO{\mathcal{O}}

\newcommand\fp{\mathfrak{p}}
\newcommand{\fa}{\mathfrak{a}}

\newcommand{\Q}{\mathbb{Q}}

\newcommand{\Z}{\mathbb{Z}}

\newcommand{\F}{\mathbb{F}}

\newcommand\Gal{{\mathrm {Gal}}}

\newcommand{\Cl}{C\ell}
\DeclareMathOperator{\lcm}{lcm}
\DeclareMathOperator{\rank}{rank}
\DeclareMathOperator{\sgn}{sgn}
\DeclareMathOperator{\Nm}{Nm}
\DeclareMathOperator{\li}{li}
\newtheorem{thm}{Theorem}[section]
\newtheorem{theorem}[thm]{Theorem}
\newtheorem{corollary}[thm]{Corollary}
\newtheorem{conj}[thm]{Conjecture}
\newtheorem{proposition}[thm]{Proposition}
\newtheorem{lemma}[thm]{Lemma}
\theoremstyle{definition}
\newtheorem{definition}{Definition}
\theoremstyle{remark}
\newtheorem{remark}{Remark}

\title[Low degree extensions with Cyclic class group]{Low degree extensions with  Cyclic class group}

\author{Sunil Kumar Pasupulati }
\address{The Institute of Mathematical Sciences, Chennai, 
	IV Cross Road, CIT Campus, Taramani, Chennai - 600 113,
	Tamil Nadu, India.}
\email{ sunilkp@imsc.res.in} 
\author{ SriLakshmi Krishnamoorthy }
\address{ Srilakshmi Krishnamoorthy
	\newline
	INDIAN INSTITUTE OF SCIENCE EDUCATION AND RESEARCH, THIRUVANANTHAPURAM, INDIA.}
\email{srilakshmi@iisertvm.ac.in}



\begin{document}
	
	\begin{abstract}
		Lenstra introduced the notion of the Euclidean ideal class, a generalization of the Euclidean domain that captures cyclic class groups. In this article, we establish the existence of Euclidean ideal classes in abelian quartic fields. As a corollary, we demonstrate that certain biquadratic fields with class number two possess a Euclidean ideal class. Additionally, we investigate the presence of Euclidean ideal classes in specific cubic and quadratic extensions.   
	\end{abstract}
	\maketitle
	\section{Introduction and Results}
	Let us consider a number field denoted as $K$, with $\mathcal{O}_K$ representing the number ring of $K$ and $\mathcal{O}_K^\times$ denotes the units of $\cO_K$. The class group, denoted as $\Cl_K$, is defined as the quotient group $J_K/P_K$, where $J_K$ is the group of fractional ideals, and $P_K$ is the group of principal fractional ideals in $K$. We use $H(K)$ to represent the Hilbert class field of $K$. Now, let  $K/\mathbb{Q}$ be an abelian extension,   then the conductor of  $K$, denoted as $f(K)$, is defined as the smallest natural number for which $K$ is a subset of $\mathbb{Q}(\zeta_{f(K)})$.
	It is also worth noting that the conductor of the Hilbert class field H(K) remains equal to $f(K)$ whenever the extension $H(K)/\mathbb{Q}$ is abelian. Additionally, we use $K_1K_2$ to represent the compositum of the number fields $K_1$ and $K_2$.
	
	The relationship between Euclidean ideals and the class group has its origins in the work of Weinberger. In 1973, assuming a generalized Riemann hypothesis (GRH), Weinberger \cite{WP}, demonstrated that for a number field $K$ where the rank of $\mathcal{O}_K^\times$ is at least 1, the number ring $\mathcal{O}_K$ is  a Euclidean domain if and only if  the class group $\Cl_K$ is trivial.
	In 1979, Lenstra extended the concept of Euclidean domains by introducing the idea of Euclidean ideal classes, as defined in \ref{EI}, to specifically  capture  cyclic class groups. As detailed in \cite{len},  he established that for a number field $K$ with a rank of $\mathcal{O}_K^\times$ greater than or equal to 1, the number ring $\mathcal{O}_K$ contains a Euclidean ideal if and only if the class group $\Cl_K$ is cyclic, assuming GRH holds.
	
	\begin{definition} [Lenstra \cite{len}] \label{EI}	Let $R$ be a Dedekind domain and $\I$ be the set of non-zero integral ideals of $R$. The ideal  $C \in \I$   is called a  {\it Euclidean ideal} if there exists a function $\Psi:\I \to W$, where $W$ is  a well-ordered set, such that for every $I \in \I $ and  $x\in I^{-1}C \setminus C$,  there exists a  $y\in C$ such that 
		\begin{align*}
			\Psi\left( (x-y)IC^{-1} \right) < \Psi (I).
		\end{align*}
		We say  $\Psi$ is a Euclidean map for $C$.
		If $C$ is a Euclidean ideal, then every ideal in the ideal class $[C]$ is also a Euclidean ideal, and the ideal class $[C]$ is called  a {\it Euclidean ideal class.}
	\end{definition}
	The Euclidean ideal class is equivalent to the principal ideal generated by $1_R$ if and only if the ring $R$ is a Euclidean domain. Lenstra \cite{len} introduced the concept of Euclidean ideal classes in a more general setting for integral domains. However, he showed that if an integral domain contains a Euclidean ideal class, it inherently becomes a Dedekind domain with a cyclic class group.
	In this paper, we explore the presence of ideal classes within the rings of integers in number fields.
	
	Ram Murty and Graves in \cite{RG13}, managed to circumvent the need for assuming the Generalized Riemann Hypothesis (GRH) in Lenstra's result. They successfully established the existence of a Euclidean ideal class   in certain abelian number fields.
	\begin{theorem}[Ram Murty and H.Graves \cite{RG13}]
		Let $K$ be a number field that is Galois over $\Q$. If its Hilbert class field $H(K)$ has an abelian Galois group over $\Q$ and if $\rank(\cO_K ^{\times})\geq 4$, then
		$$\Cl_K = \langle [C]	\rangle \  \text { if and only if }\ [C] \ \text{is a Euclidean ideal class.} $$
	\end{theorem}
	Deshouillers, Gun, and Sivaraman \cite{DGS}  proved the existence of the Euclidean ideal class in number fields with  $\rank(\cO_K^{\times} )\geq 3$  with an additional assumption that the Galois group $ \Gal\left(\Q(\zeta_{f(K)})/K\right)$ is cyclic. They proved the following theorem using the linear sieve with the error term given by  Iwaniec \cite{errorterm80}. 
	\begin{theorem} [Deshouillers, Gun, and Sivaraman \cite{DGS}]\label{Jh}
		Let $K$ be a number field such that rank  $\left(\cO_K ^{\times}\right)\geq 3$  and the Hilbert class field $H(K)$ is abelian over $\Q$. Also suppose that the conductor of $H(K)$ is $f(K)$ and $\Q(\zeta_{f(K)})$ over $K$ is cyclic. Then $K$ has a Euclidean ideal class.
	\end{theorem} 
	\begin{remark}
		Let us consider the example of a number field, $K$, given as $K=\Q(\sqrt{11},\sqrt{247})$. Despite the fact that the rank of $\mathcal{O}_K^\times$ is $3$ in this real quartic field $K$, we cannot conclude that $K$ possesses a Euclidean ideal by applying Theorem \ref{Jh}. This is because the Galois group $\Gal(\Q(\zeta_f)/K) \cong \Z/10\Z\times \Z/12\Z\times \Z/9\Z$ is not cyclic, where $f=10808$ represents the conductor of the Hilbert class field of $K$. As a result, one of the assumptions in Theorem \ref{Jh} is not met. However, we can indeed ascertain that $K$ has a Euclidean ideal by utilizing Theorem \ref{thm41}.
		
	\end{remark}
	We prove the existence of a Euclidean ideal class in abelian real quartic fields with a prime class number without the assumption that $\Q(\zeta_f)$ over $K$ is cyclic. We also prove    Corollary \ref{corq} that almost settles the discussion on the existence of Euclidean ideal classes in certain biquadratic fields with class number two.

	\begin{theorem}\label{thm41}
		Let $K$ be a real quartic extension with prime class number and abelian (over $\Q$)  Hilbert class field.
		Define  
		$f$ as the least common multiple of $16$ and the conductor of $K$. Let $G$ to denote the Galois group of $\mathbb{Q}(\zeta_f)$ over $K$ and $G_{\ell}$,  the Galois group of
		$\Q(\zeta_f) $ over $\Q(\zeta_{\ell})$.  If 
		\begin{align*}
			G \not \subseteq 	\bigcup_{\ell} G_{\ell}\bigcup \Gal\left({\Q(\zeta_f)}/{H(K)}\right),
		\end{align*}
		where $\ell$ is an odd prime dividing $f$ or $\ell=4$, then $K$ has a Euclidean ideal class.
	\end{theorem}
    The study of the Euclidean ideals in biquadratic fields caught many eyes, for example, look at \cite{G1,CH,JS19,sp3}. The following corollary generalizes the main results of \cite{CH},~\cite{JS19}, and   \cite{sp3}.
\begin{corollary}\label{corq}	
		Let  $q,r,s$ be distinct primes and  $K=\Q\left(\sqrt{q},\sqrt{rs}\right)$ be a totally real biquadratic field with $h_K=2$. If one of the following conditions is satisfied:  
\begin{enumerate}
    \item $q \equiv 3 \pmod{4}$
    \item At least one of $r$ or $s$ is $1 \pmod{4}$,
\end{enumerate}
then $C\ell_K$ has a Euclidean ideal class. 	
\end{corollary}

We have examined the existence of the Euclidean ideal classes in the biquadratic field  $K_2=\Q\left(\sqrt{2},\sqrt{qr}\right)$.
The following corollary will remove the assumption that both the primes  $q$ and $r$ are congruent to $1\mod4$ from  \cite[Theorem 1.5]{JS19}.
\begin{corollary}\label{corprime2}
	Let $q,r$ be odd primes. If at least one of 
	$q$ or $r$ is congruent to $1$ modulo $4$, then $K_2=\Q(\sqrt{2},\sqrt{qr})$ has a Euclidean ideal class whenever $C\ell_K$ has order 2. 
\end{corollary}
	
		\begin{corollary}\label{cor32}
		Let $q,r$ be primes such that $\Q(\sqrt{q}) $ and $\Q(\sqrt{r})$ have class number one and $\Q(\sqrt{qr})$ have class number $2p$, where p is prime. If all the quadratic subfields of the biquadratic number field $K=\Q(\sqrt{q},\sqrt{r})$,  have deficiency $0$  and   $K$ has deficiency $1$ (see \Cref{deficiency}), then $K$ has a  Euclidean ideal.
	\end{corollary}
	
	The article is structured as follows: In Section~\ref{sec:prelims}, we provide essential preliminaries that will aid in establishing our main results. Section~\ref{sec:quartic} is dedicated to the proof of \Cref{thm41}. In Section~\ref{sec:biquadratic}, we delve into the examination of the existence of Euclidean ideals in biquadratic extensions  and subsequently prove  \Cref{corq},\Cref{corprime2}, and \Cref{cor32}.
	Section~\ref{sec:cubic-quadratic} focuses on exploring the presence of Euclidean ideals in both cubic and quadratic extensions.
	In Section~\ref{sec:conj} we discuss a consequence of the Elliott-Halberstam conjecture. Finally, in the last section, we discuss some examples.

	\section{A few useful preliminaries and  results }\label{sec:prelims}
	\begin{definition}
		Let $K$ be a number field. The Hilbert class field $H(K)$ of $K$ is the maximal unramified abelian extension of $K$.
	\end{definition}
	\begin{lemma}\label{bl1}
		The conductor  of the quadratic field $K=\Q(\sqrt{d}) $  is  
		\begin{align*}
			f(K)=\begin{cases}
				d, &  \text{if} ~ d\equiv 1 \mod{4},\\
				4d, & \text{if} ~ d\equiv 2,3 \mod{4}.
			\end{cases}
		\end{align*}
	\end{lemma}
	\begin{lemma}\label{bl2}
		Let $L$ be  the compositum of $K_1$ and $K_2$, then the  conductor of the number field $L$,  $f(L)$ is  equal to  $\lcm(f(K_1),f(K_2))$.
	\end{lemma}
	\begin{lemma}\label{comp}
		Let $K_1,K_2$  be Galois extensions of $K$. The compositum of $K_1$ and $K_2$ is a Galois extension of $K$ and  if $K_1,K_2$  are linearly disjoint extension of $K$,  then we have 
		\begin{align*}
			\Gal(K_1K_2/K) \cong  \Gal(K_1/K)\times \Gal(K_2/K).
		\end{align*}
	\end{lemma}

Let $K$ be a totally real number field of degree n, then  define the Archimedean signature map  
$	\sgn_\infty :K^{\times} \to \prod_{\nu \mid \infty } \{\pm1\} \cong  \F^n_2$
as the surjective group homomorphism recording the signs of elements of  $K^\times$  under each real embedding.
\begin{definition}
	The {\bf unit signature rank }of $K$ is defined to  be the dimension  of $\sgn_\infty \left(\cO_K^{\times}\right)$ as a subspace of  $\F_2^n$.
\end{definition}
\begin{definition}\label{deficiency}
	The  {\bf deficiency} of $K$, also known as a deficiency in the unit signature rank, is denoted by the symbol $\delta (K)$, and it is the corank of $\sgn_\infty \left(\cO_K^{\times}\right)$. More specifically, it equals $n$ minus the signature rank of the units of $K$.
\end{definition}

	Graves \cite{growth}  proved a useful growth result that gives a condition for the existence of a Euclidean ideal in number fields without the assumption of GRH.
	\begin{theorem}[Graves \cite{growth}]
		Suppose that $K$ is a number field such that $\rank(\cO_K ^{\times})  \geq 1$
		and the ideal $C$ is a non-zero ideal of $\cO_K$. If $[C]$  generates the class group of $K$ and
		\[\Big| \left\{ \text{Prime\ ideal}  \  \fp \subseteq  \cO_K~ \Big \vert~\Nm(\fp)\leq X, \ [\fp]=[C], \   \cO_K^{\times}\twoheadrightarrow (\cO_K/\fp) ^{\times} \right\}\Big|\gg \frac{X}{(\log X)^2},
		\]
		then $[C]$ is a Euclidean ideal class.
	\end{theorem}
	
	Gun and Sivaraman proved the following generalized version of the above theorem.   
	\begin{theorem}[Gun and Sivaraman \cite{GS20}]\label{sgl}
		Suppose that $K$ is a number field, $ \rank (\cO_K ^{\times})\geq 1$  and $\Cl_K=\langle [\fa]\rangle$. If there exists an unbounded increasing sequence $\{X_n\}_{n\in \N}$ such that 
		\begin{align*}
			\Bigg|\Big\{\text{Prime ideal} \ \fp \subseteq \ \cO_K  \Big \vert[\fp]=[\fa],\Nm(\fp)\leq X_n, \cO_K ^{\times}\twoheadrightarrow (\cO_K/\fp) ^{\times}\Big\} \Bigg|\gg \frac{X_n}{\log^2X_n}, 
		\end{align*}
		then $[\fa]$ is a Euclidean ideal class.
	\end{theorem}
	We state a crucial lemma that will be used in the proofs of the main theorems. 	
	\begin{lemma}[Heath-Brown \cite{HB}]\label{HBL}
		Suppose that $u$ and $v$ are natural numbers with the following properties,
		\[ (u,v)=1,\quad v\equiv 0 \pmod{16}, \quad  \text{and}  \quad  \left( \frac{u-1}{2},v\right)=1. \]
		Then there exist  $b, c \in (1/4,1/2)$ with $b<c$ such that for any $\epsilon>0$, 
		\begin{align*}
			\Bigg| & \Big\{ p \ \text{is a prime} ~  ~\Big \vert  ~ p\equiv u\pmod{v},~ p\in(X^{1-\epsilon},X)
			~	\text{such that} \ ~\frac{p-1}{2} \  ~\text{is either a} \\  &\text{prime or product of two primes}\ q_1q_2 \ with \ X^b<q_1<X^c\Big\}\Bigg|\gg \frac{X}{\log ^2X}.
		\end{align*}
	\end{lemma}
	Now we end this section by stating a famous conjecture in Sieve Theory.
	\begin{conj}
		[Elliott-Halberstam conjecture \cite{EH}]  Let $a,q$ be natural numbers,  $\phi$ be the Euler totient function, $\pi(Y,q,a):=\Big\{ p\leq Y | \  p \ \text{is prime,} \, p\equiv a \mod{q} \Big\}$, and\\$\mathrm{li}(Y):=\int_{2}^Y \frac{1}{\log t}dt$. For every real number $\theta <1$ and for every positive integer $k>0$ and for all real numbers $X>2$, we have
		\begin{align*}
			\sum_{q\leq X^\theta} 	\mathop{max}_{Y\leq X} \mathop{max}_{(a,q)=1}\left|  \pi(Y,q,a)-\frac{\li(Y)}{\phi(q)}\right| \ll \frac{X}{\log^k X}.
		\end{align*}
		
	\end{conj}
	
	\section{Euclidean ideals in  Real  Quartic extensions}\label{sec:quartic}
	Let $K$ be a  real quartic number field such that both $K$ and $ H(K)$ are abelian over $\Q$.   Let $f$ be the least common multiple of $16$ and  $ f(K)$. Also let  $G_{\ell}$ be  Galois group of  $\Q(\zeta_f)$ over $\Q(\zeta_{\ell})$.

	\begin{lemma}
		Suppose that the Galois group $G$ of $\Q(\zeta_f)$ over $K$ satisfies the hypothesis of Theorem \ref{thm41}. Then there exists an integer $a$  such that any prime $p$ with $p\equiv a \pmod{f}$    splits completely in $K$ but does not split completely in $H(K)$. Furthermore, there exist $b$ and $c$  in $(1/4,1/2)$ such that for any $X$ and  $\epsilon>0$, we have 
		\begin{align*}
			\Bigg| \Big\{ p \ ~\text{is prime}  ~\Big \vert ~ p\equiv a\pmod{f} , & ~ p\in(X^{1-\epsilon},X)
			~\text{such that} \ ~\frac{p-1}{2} ~\text{is either a }  ~\text{prime or product of} \\ &  \text{two primes}\  q_1q_2 \ with \ X^b<q_1<X^c\Big\}\Bigg|\gg \frac{X}{\log ^2X}.
		\end{align*}
		
	\end{lemma}
	\begin{proof}
		According to the  hypothesis  of Theorem \ref{thm41}, we have 
		\begin{align*}
			G \not \subseteq \bigcup _l G_{\ell}\bigcup \Gal\left({\Q(\zeta_f)}/{H(K)}\right).
		\end{align*} Let $\sigma \in  G \setminus \left(\bigcup_l G_{\ell}\bigcup\Gal\left({\Q(\zeta_f)}/{H(K)}\right) \right)$.  
		Choose $a \pmod{f}$ to be the image of  $\sigma$  under the isomorphism $\Gal (\Q(\zeta_f)/\Q) \to \left(\Z/f\Z\right) ^{\times}$. By the properties of  the Artin symbol, whenever $p\equiv a \pmod{f} $ then 
		\begin{align*}
			\left(\frac{K/\Q}{p}\right)=\sigma|_K=id \ \text{ and}\  \left(\frac{H(K)/\Q}{p}\right)=\sigma|_{H(K)}\neq id.	
		\end{align*}
		Therefore $p$ splits completely in $K$ but does not split completely in $H(K)$. The fact that  $\sigma\not \in\bigcup G_{\ell} $ for all prime $\ell$ dividing $f$  translates to $ \left(\frac{a-1}{2},f\right)=1$.
		By applying  Lemma \ref{HBL}  for $u=a$ and $v=f$, there exist $b,c \in \left(1/4,1/2\right)$ and for every $\epsilon\geq 0$,
		we get 
		\begin{align*}
			\Bigg|\Big\{~ p \ ~\text{is prime} ~\Big \vert ~ &  p\equiv a\pmod{f} , ~p\in(X^{1-\epsilon},X)
			~\text{such that} \ ~\frac{p-1}{2} ~\text{is either a } \\  &\text{prime or product of two primes}\ q_1q_2 \ with \ X^b<q_1<X^c\Big\}\Bigg|\gg \frac{X}{\log ^2X}.
		\end{align*}
	\end{proof}
	For the number field $K$ in Theorem \ref{thm41}, we define 
	\begin{align*}
		M_{\epsilon}:=&\Big\{\fp\subseteq \cO_K  \ \text{is a prime ideal} \  ~\Big \vert\ \Nm(\fp)=p\ \text{ is  a prime,} \ ~ p\equiv a\pmod{f},\  \frac{p-1}{2} \ \text{is}\\  & \text{ either  a prime or product of two primes}\ q_1q_2 \ with \ p^b<q_1<p^{\frac{c}{1-\epsilon}}   \Big\},\quad  \text{and}
	\end{align*} 
	$$M_{\epsilon}(X)=\{\fp\in M_{\epsilon} ~\Big \vert\Nm(\fp)\leq X \}. $$
	
	\begin{lemma}\label{lm2}
		Let $K$ be a totally real number field, and let $e_1, e_2$, and $e_3$ be non-zero elements of $K$ that are multiplicatively independent. In this case, for some $i \in \{1,2,3\}$,  either $e_i$ or $-e_i$ serves as a primitive root modulo $\fp$ for infinitely many prime ideals in the set $M_{\epsilon}$.
		Let us denote this set of prime ideals as $P$, and let $P(X)$ represent the elements in $P$ whose norm is less than or equal to $X$. 
		Then, there exists an increasing and unbounded sequence $\{x_n\}_{n\in \mathbb{N}}$ such that
		
		$$P(x_n)\gg \frac{x_n}{\log^2x_n}. $$
	\end{lemma} 
	\begin{proof}
		We refer the reader to   \cite[Lemma 16]{GS20}.
	\end{proof}
	Consider a real quartic field $K$ where $\mathcal{O}_K^{\times}$ contains three multiplicatively independent elements, denoted as $\epsilon_1, \epsilon_2, \epsilon_3$.   By Lemma \ref{lm2}, there exists an element $\eta$ in the set $\{\pm \epsilon_1, \pm \epsilon_2, \pm \epsilon_3\}$ and a corresponding set of prime ideals $P$. For every prime ideal $\mathfrak{p}$ in $P$, $\eta$ acts as a primitive root modulo $\mathfrak{p}$.
	Furthermore, we get an increasing and unbounded sequence $\{x_n\}_{n\in\mathbb{N}}$ such that
	\begin{align*}
		\biggr | P(x_n)\biggr| \geq \frac{cx_n}{\log^2x_n}.
	\end{align*}
	Since there are only finitely many ideal classes in the class group $\Cl_K$ of $K$, we can assert that there exists a non-trivial ideal class $[\mathfrak{a}]$  and  a subsequence $\{y_n\}_{n\in\mathbb{N}}$  of $\{x_n\}_{n\in\mathbb{N}}$ such that
	\begin{align*}
		\Bigg| \Big\{ \fp \in [\fa]~  \Big \vert~\fp\in P,~\Nm(\fp)\leq y_n \Big\}\Bigg|\gg \frac{y_n}{h_k\log^2y_n}.
	\end{align*}
	Otherwise, we would have  \[\biggr | P(x_n)\biggr| =o\left( \frac{x_n}{\log^2x_n}\right).\]
	
	\noindent
	Since every prime ideal $\mathfrak{p} \in M_{\epsilon}$  lies over a prime $p$ congruent to $a$ modulo $f$, it follows that $\mathfrak{p}$ does not belong to the trivial ideal class. Given that  $\Cl_K$ is a cyclic group of prime order,  $[\mathfrak{a}]$  must be a generator of 
	$\Cl_K$. According to Theorem \ref{sgl}, we can conclude that $[\mathfrak{a}]$ is the Euclidean ideal class.
	This completes the proof of  \Cref{thm41}$.\square$

	\section{Euclidean ideals in biquadratic fields}\label{sec:biquadratic}
	
	\begin{proposition}\label{Hilbertfield}
		Let $K=\Q\left(\sqrt{q},\sqrt{rs}\right) $, where  $q,r$, and $s$ are odd primes. 
        
        If one of the following conditions is satisfied:  
\begin{enumerate}
    \item $q \equiv 3 \pmod{4}$
    \item At least one of $r$ or $s$ is $1 \pmod{4}$,
\end{enumerate}
        then  \begin{align*}
			\Q\left(\sqrt{q},\sqrt{r},\sqrt{s} \right) \subseteq H(K).
		\end{align*}
	\end{proposition}
	\begin{proof}
		Let $L$ be the number field $ \mathbb{Q}(\sqrt{q}, \sqrt{r}, \sqrt{s})$. It can be observed that $L$ is an abelian extension over $K = \mathbb{Q}(\sqrt{q}, \sqrt{rs})$.  It is enough to prove that all primes of $K$ are unramified in $L$. To demonstrate that all primes of $ K$  are unramified in the number field $L$, it is enough to examine the primes that lie above $q, r, s$, and $2$. Furthermore, it can be observed that 
		\begin{align*}
		L=K\Q(\sqrt{qr})=K\Q(\sqrt{sq})=K\Q(\sqrt{r})=K\Q(\sqrt{s}).
		\end{align*} 
		Therefore,  by  \cite[Lemma 2.4.]{sp3},   we have  $L\subseteq H(K)$,
		where $H(K)$ denotes the Hilbert class field of $K$.
	\end{proof}
	By the Lemma~\ref{bl1} and  Lemma~\ref{bl2}, we know that $f(K) = 4qrs$. Let us define $f$ as the least common multiple of $f(K)$ and $16$. In other words, we have $f = \operatorname{lcm}(f(K), 16) = 16qrs$.
	
	\begin{lemma}\label{bl3}
		Let $q,r,s$ be prime numbers satisfying one of the following conditions:
\begin{enumerate}
    \item $q \equiv 3 \pmod{4}$,
    \item At least one of $r$ or $s$ is $1 \pmod{4}$.
\end{enumerate}  Suppose that the number field $K=\Q(\sqrt{q},\sqrt{rs})$   has class number two.    If  $\Gal(H(K)/K)=\langle \sigma \rangle$,  then there exists a lift  $\hat \sigma \in 	\Gal\left(\Q(\zeta_f)/K  \right)$   of $\sigma$   such that 
		\begin{align*}
			\hat \sigma \not \in \bigcup _{\ell} \Gal\left(\Q(\zeta_f)/\Q(\zeta_{\ell}) \right)\bigcup \Gal\left({\Q(\zeta_f)}/{H(K)}\right)
		\end{align*}
	where $\ell$ runs over the odd prime divisors of $f$  and  $\ell=4$. 
	\end{lemma}
	\begin{proof}
		
		By  Proposition \ref{Hilbertfield}, the Hilbert class field $H(K)$ is  given by $\mathbb{Q}\left(\sqrt{q}, \sqrt{r}, \sqrt{s}\right)$. Moreover, $\operatorname{Gal}(H(K)/K) = \langle \sigma \rangle$, where 
        \[ \sigma(\sqrt{r}) = -\sqrt{r}, \sigma(\sqrt{s}) = -\sqrt{s}, ~\text{and} \ \sigma(\sqrt{q}) = \sqrt{q}.\]
		
		\subsection*{Case (i): Both $r$ and $s$ are congruent to $1 \pmod{4}$}  
Let $\tau$ be any lift of $\sigma$ to $\mathbb{Q}(\zeta_f)$ then  $\tau(\sqrt{r}) = -\sqrt{r}$ and $\tau(\sqrt{s}) = -\sqrt{s}$. Therefore, $\tau$ does not belong to $\operatorname{Gal}(\mathbb{Q}(\zeta_f)/\mathbb{Q}(\zeta_{\ell}))$ for $\ell = r, s$.

{\bf Subcase (i) $q \equiv 1 \pmod{4}$}:  
Consider the field $F$, to be  the compositum of $H(K)$ and $\mathbb{Q}(\zeta_q)$. In this case, the intersection $H(K) \cap \mathbb{Q}(\zeta_q)$ equals $\mathbb{Q}(\sqrt{q})$. Therefore, by Lemma~\ref{comp}, we have
\[
\Gal(F/\mathbb{Q}(\sqrt{q})) \cong \Gal(H(K)/\mathbb{Q}(\sqrt{q})) \times \Gal(\mathbb{Q}(\zeta_q)/\mathbb{Q}(\sqrt{q})).
\]
As a result, there exists an automorphism $\sigma_1 \in \Gal(F/K)$ lifting $\sigma$ such that $\sigma_1(\zeta_q) \neq \zeta_q$.
Next, observe that $F \cap \mathbb{Q}(i) = \mathbb{Q}$. This allows us to further lift $\sigma_1$ to an automorphism $\sigma_2 \in \Gal(F(i)/K)$ satisfying $\sigma_2(i) = -i$.
Finally, this construction yields a lift $\hat{\sigma} \in \Gal(\mathbb{Q}(\zeta_f)/K)$ of $\sigma$ such that
\[
\hat{\sigma} \notin \Gal(\mathbb{Q}(\zeta_f)/\mathbb{Q}(\zeta_{\ell}))
\]
for $\ell = 4, q, r$, and  $s$.\\
{\bf Subcase (ii): $q \equiv 3 \pmod{4}$}  
In this case, let $F$ denote the compositum of $H(K)$ and $\mathbb{Q}(i)$. Since $H(K) \cap \mathbb{Q}(i) = \mathbb{Q}$, there exists an automorphism $\sigma_1 \in \Gal(F/\mathbb{Q})$ such that $\sigma_1(i) = -i$ and $\sigma_1|_{H(K)} = \sigma$.
Therefore, we have $\sigma_1(i\sqrt{q}) = -i\sqrt{q}$. Consequently, any lift of $\sigma_1$ does not belong to $\Gal(\mathbb{Q}(\zeta_f)/\mathbb{Q}(\zeta_{\ell}))$ for $\ell = 4, q, r$, and $ s$.

        \subsection*{Case (ii): Both $r$ and $s$ are congruent to $3 \pmod{4}$}  
In this case, we may assume that $q \equiv 3 \pmod{4}$. Let $E$ denote the compositum of the fields $\mathbb{Q}(i)$, $H(K)$, and $\mathbb{Q}(\zeta_q)$. 
By an argument analogous to that in Case (i), there exists an automorphism $\sigma_1 \in \Gal(E/K)$ such that $\sigma_1(i) \neq i$, $\sigma_1(\zeta_q) \neq \zeta_q$, and $\sigma_1|_{H(K)} = \sigma$.
Next, consider $E_1$, the compositum of $E$ and $\mathbb{Q}(\zeta_r)$. Since $E \cap \mathbb{Q}(\zeta_r) = \mathbb{Q}(i\sqrt{r})$, Lemma~\ref{comp} implies
\[
\Gal(E_1/\mathbb{Q}(i\sqrt{r})) \cong \Gal(E/\mathbb{Q}(i\sqrt{r})) \times \Gal(\mathbb{Q}(\zeta_r)/\mathbb{Q}(i\sqrt{r})).
\]
Therefore, there exists a lift $\sigma_2 \in \Gal(E_1/K)$ of $\sigma_1$.
Now, let $E_2$ be the compositum of $E_1$ and $\mathbb{Q}(\zeta_s)$. Since $E_2 \cap \mathbb{Q}(\zeta_s) = \mathbb{Q}(i\sqrt{s})$, Lemma~\ref{comp} gives
\[
\Gal(E_2/\mathbb{Q}(i\sqrt{s})) \cong \Gal(E_1/\mathbb{Q}(i\sqrt{s})) \times \Gal(\mathbb{Q}(\zeta_s)/\mathbb{Q}(i\sqrt{s})).
\]
Thus, there exists a lift $\sigma_3 \in \Gal(E_2/K)$ of $\sigma_2$ such that $\sigma_3(\zeta_s) \neq \zeta_s$.
In conclusion, this process yields a lift $\hat{\sigma} \in \Gal(\mathbb{Q}(\zeta_f)/K)$ of $\sigma$ with the property that
\[
\hat{\sigma} \notin \Gal(\mathbb{Q}(\zeta_f)/\mathbb{Q}(\zeta_{\ell}))
\]
for $\ell = 4, q, r$, and  $s$.

        \subsection*{Case (iii): Exactly one of $r$ or $s$ is congruent to $1$ modulo $4$}  
Without loss of generality, we may assume that $r \equiv 1 \pmod{4}$ and $s \equiv 3 \pmod{4}$. 
Since $\sigma(\sqrt{r}) = -\sqrt{r}$ and $\sqrt{r} \in \mathbb{Q}(\zeta_r)$, it follows that, for any lift $\tau$ of $\sigma$, we must have $\tau(\zeta_r) \neq \zeta_r$.
Now, let $L$ denote the compositum of $H(K), \mathbb{Q}(i)$, and $\mathbb{Q}(\zeta_q)$. By an argument similar to that in Case (i), there exists an automorphism $\tau_1 \in \Gal(L/K)$ such that $\tau_1|_{H(K)} = \sigma$ and $\tau_1 \notin \Gal(L/\mathbb{Q}(\zeta_{\ell}))$ for $\ell = 4, q$.
Next, consider $L_1$, the compositum of $L$ and $\mathbb{Q}(\zeta_s)$. Since $L_1 \cap \mathbb{Q}(\zeta_s) = \mathbb{Q}(i\sqrt{s})$, Lemma~\ref{comp} provides the isomorphism
\[
\Gal(L_1/\mathbb{Q}(i\sqrt{s})) \cong \Gal(L/\mathbb{Q}(i\sqrt{s})) \times \Gal(\mathbb{Q}(\zeta_s)/\mathbb{Q}(i\sqrt{s})).
\]
Consequently, there exists a lift $\tau_2 \in \Gal(L_1/K)$ of $\tau_1$ such that $\tau_2(\zeta_s) \neq \zeta_s$.
Finally, let $\hat{\sigma}$ be a lift of $\tau_2$ to $\mathbb{Q}(\zeta_f)$. By construction, we have
\[
\hat{\sigma} \notin \Gal\big( \mathbb{Q}(\zeta_f)/\mathbb{Q}(\zeta_{\ell}) \big)
\]
for each $\ell=4,q,r$, and $s$.
\end{proof}	

	\subsection{ Proof of  \Cref{corq}}
\begin{proof}
Let $q$, $r$, and $s$ be distinct primes satisfying the conditions of \Cref{corq}. Suppose $K$ denotes the biquadratic field $K = \mathbb{Q}\left(\sqrt{q}, \sqrt{rs}\right)$. By Theorem~\ref{Hilbertfield}, the Hilbert class field of $K$ is given by
\[
H(K) = \mathbb{Q}\left( \sqrt{q}, \sqrt{r}, \sqrt{s} \right).
\]
Now, by Lemma~\ref{bl3}, we have
\[
\Gal\left( \mathbb{Q}(\zeta_f) / K \right) \not\subseteq 
\bigcup_{\substack{\ell \mid f \\ \ell\ \text{prime}}} \Gal\left( \mathbb{Q}(\zeta_f) / \mathbb{Q}(\zeta_{\ell}) \right) 
\cup \Gal\left( \mathbb{Q}(\zeta_f) / H(K) \right).
\]
Therefore, by Theorem~\ref{thm41}, the biquadratic field $K$ possesses a Euclidean ideal class.
\end{proof}

\subsection{Proof of \Cref{corprime2}}

\begin{proof}
Let $K_2$ be the biquadratic field of the form $K_2 = \mathbb{Q}(\sqrt{2}, \sqrt{qr})$. Without loss of generality, we may assume that $q \equiv 1 \pmod{4}$. 
By arguments analogous to those in \Cref{Hilbertfield}, it follows that the Hilbert class field of $K_2$ is given by
\[
H(K_2) = \mathbb{Q}(\sqrt{2}, \sqrt{q}, \sqrt{r}).
\]
Moreover, by Lemmas~\ref{bl1} and~\ref{bl2}, the conductor of $K_2$ is $f(K_2) = 8qr$. 
Applying Lemma~\ref{bl3}, we obtain
\[
\Gal\left( \mathbb{Q}(\zeta_{f(K_2)}) / K_2 \right) 
\not\subseteq 
\bigcup_{\substack{ \ell \mid f(K_2) \\ \ell\ \text{prime}}} 
\Gal\left( \mathbb{Q}(\zeta_{f(K_2)}) / \mathbb{Q}(\zeta_{\ell}) \right) 
\cup 
\Gal\left( \mathbb{Q}(\zeta_{f(K_2)}) / H(K_2) \right).
\]
Therefore, by Theorem~\ref{thm41}, the biquadratic field $K_2$ possesses a Euclidean ideal class.
\end{proof}

\subsection{Proof of \Cref{cor32}}

\begin{proof}
Let $K_1 = \mathbb{Q}(\sqrt{q}), K_2 = \mathbb{Q}(\sqrt{r})$, and $K_3 = \mathbb{Q}(\sqrt{rq})$. By \cite[Theorem 3.1]{SP95}, the class number $h_K$ of the biquadratic field $K$ satisfies
\[
h_K = \frac{\left[ \mathcal{O}_K^\times : \mathcal{O}_{K_1}^\times \mathcal{O}_{K_2}^\times \mathcal{O}_{K_3}^\times \right]}{4} \cdot h_{K_1} h_{K_2} h_{K_3}.
\]
Since all quadratic subfields of $K$ have deficiency zero, each possesses a fundamental unit with negative norm. By \cite[Proposition 1]{dumithershy23}, the unit group of $K$ is either
\[
\{ -1, \epsilon_1, \epsilon_2, \epsilon_3 \}
\quad \text{or} \quad
\{ -1, \epsilon_1, \epsilon_2, \sqrt{\epsilon_1 \epsilon_2 \epsilon_3} \},
\]
where $\epsilon_i$ denotes the fundamental unit of $K_i$ for $i = 1, 2, 3$. 
Furthermore, as $K$ has deficiency $1$, By  \cite[Proposition 7]{dumithershy23}, we have
$\mathcal{O}_K^\times = \langle -1, \epsilon_1, \epsilon_2, \epsilon_3 \rangle$.
It follows that
$
\left[ \mathcal{O}_K^\times : \mathcal{O}_{K_1}^\times \mathcal{O}_{K_2}^\times \mathcal{O}_{K_3}^\times \right] = 2$.
Therefore, the class number of $K$ is equal to $p$.
To apply Theorem~\ref{thm41}, it remains to show that
\[
\Gal\left( \mathbb{Q}(\zeta_f) / K \right)
\not\subseteq
\bigcup_{\substack{ \ell \mid f \\ \ell\ \text{odd prime or } \ell = 4}} 
\Gal\left( \mathbb{Q}(\zeta_f) / \mathbb{Q}(\zeta_{\ell}) \right)
\cup
\Gal\left( \mathbb{Q}(\zeta_f) / H(K) \right).
\]
Since the class number of $K$ is prime, the Galois group $\Gal(H(K)/K)$ is cyclic. Let $\sigma$ denote a generator of $\Gal(H(K)/K)$. Then any lift $\overline{\sigma}$ of $\sigma$ to $\mathbb{Q}(\zeta_f)$ does not lie in $\Gal(\mathbb{Q}(\zeta_f)/H(K))$.
For any prime $\ell$ dividing $f$, if $\mathbb{Q}(\zeta_{\ell}) \cap H(K) \neq \mathbb{Q}$, there exists a lift of $\sigma$ to $\mathbb{Q}(\zeta_f)$ that does not fix $\zeta_{\ell}$. On the other hand, if $\mathbb{Q}(\zeta_{\ell}) \cap H(K) = \mathbb{Q}$, then by Lemma~\ref{comp}, we can always find a lift of $\sigma$ that acts nontrivially on $\zeta_{\ell}$. This establishes the required condition, and the result follows.
\end{proof}

	\section{Euclidean ideals in  real cubic and quadratic fields}\label{sec:cubic-quadratic}

	Gun and Sivaraman \cite{GS20} also examined the existence of a Euclidean ideal in the case of abelian number fields $K$ with a prime class number and  $ \text{rank}(\cO_K ^{\times}) \leq 2$.  They proved the following two theorems.
	\begin{theorem}[Gun and Sivaraman \cite{GS20}] ] \label{thm:GScubic}
		Let $K_1, K_2$ be distinct real cubic fields with prime class number having conductors $f_1,f_2$ respectively. And also let $H(K_i)$ be abelian over $\Q$  for $i=1,2$ and $f =lcm (f_1,f_2,16)$.    If 
		\begin{align*}
			\Gal\left(\Q(\zeta_f)/K_1K_2\right)\not \subseteq  \bigcup_ {\ell\mid f}  Gal\left({\Q(\zeta_f)}/{\Q(\zeta_{\ell})} \right) \bigcup_{i=1}^2 \Gal\left(\Q(\zeta_f)/ H(K_i)\right),
		\end{align*}
		where $\ell$ is an odd prime or $\ell=4 $,  then at least  one of  $K_1,K_2$ has a Euclidean ideal class.
	\end{theorem}

	\begin{theorem}[Gun and Sivaraman \cite{GS20}] \label{thm:GSquadratic}
		Let $K_1, K_2, K_3$ be distinct real quadratic fields with prime class number having conductors $f_1,f_2,f_3$ respectively. And also let $H(K_i)$ be abelian over $\Q$  for $i=1,2,3$ and 
		$f =lcm (f_1,f_2,f_3,16)$.    If 
		\begin{align*}
			\Gal\left(\Q(\zeta_f)/K_1K_2K_3\right)\not \subseteq  \bigcup_ {\ell\mid f}  Gal\left(\Q(\zeta_f)/\Q(\zeta_{\ell}) \right) \bigcup_{i=1}^3 \Gal\left(\Q(\zeta_f)/ H(K_i)\right),
		\end{align*}
		where $\ell$ is an odd prime  or $\ell=4 $, then at least one of  $K_1,K_2,K_3$ has  a Euclidean ideal class.
	\end{theorem}
	During our investigation of Euclidean ideals in both quadratic and cubic extensions, our primary objective was to eliminate the assumption of a prime class number in Theorems \ref{thm:GScubic} and \ref{thm:GSquadratic}. During our study, we made an intriguing observation. Although the following proposition may be well-known, we have included the proof for the sake of completeness, thus avoiding the need for a specific reference.

	\begin{proposition}\label{cyclic:classgroup}
    Let $K$ be a real abelian extension of $\mathbb{Q}$ of prime degree $p$, and suppose that its Hilbert class field $H(K)$ is also abelian over $\mathbb{Q}$. If the class group $\Cl_K$ is cyclic, then the class number of $K$ is either $1$ or $p$.
\end{proposition}

\begin{proof}
We establish this result using the genus theory of number fields, referring to Ishida’s book \cite{genustheorybook} as a key reference. Consider first the quadratic case, where $K = \mathbb{Q}(\sqrt{d})$. Let the prime factorization of the Discriminant of $K$ 
\[
D_K = p_1 p_2 \cdots p_t p_{t+1} \cdots p_{t+r},
\]
where $p_i \equiv 1 \pmod{4}$ for $1 \leq i \leq t$ and $p_{t+i} \equiv 3 \pmod{4}$ for $1 \leq i \leq r$.
Since $H(K)$ is abelian over $\Q$, it implies that H(K) is the same as the Genus field. By \cite[Chapter 4]{genustheorybook}, considering the ramification at the infinite prime for real quadratic fields, we observe that the structure of the Hilbert class field $H(K)$  can be shown as follows:
\[
H(K) =
\begin{cases}
    \mathbb{Q}(\sqrt{p_1}, \dots, \sqrt{p_t}, \sqrt{p_{t+1} p_{t+2}}, \dots, \sqrt{p_{t+1} p_{t+r}}) & \text{if $K$ is a real quadratic field,} \\
    \mathbb{Q}(\sqrt{p_1}, \dots, \sqrt{p_{t+r}}) & \text{if $K$ is an imaginary quadratic field.}
\end{cases}
\]
Since $H(K)$ is cyclic over $K$, we must have $[H(K):K] = 1$ or $p$.

By \cite[Chapters 4 and 5]{genustheorybook}, the genus field of $K$ is constructed as a composite of disjoint degree $p$ extensions of $\mathbb{Q}$ along with $K$. This implies that if $\Cl_K$ is cyclic, then its order must be either $1$ or $p$.
\end{proof}

	We derive the following corollaries from the proposition presented above.
	\begin{corollary}
		Consider two distinct real cubic fields
		$K_1$ and $K_2$ that have cyclic class groups.  Also let both fields, $K_1$ and $K_2$ be abelian extensions over  $\Q$.  Furthermore,  define  $f$ to be the least common multiple of $f(K_1),f(K_2)$, and $16$. If 
		\begin{align*}
			\Gal\left(\Q(\zeta_f)/K_1K_2\right)\not \subseteq  \bigcup_ {\ell\mid f}  Gal\left({\Q(\zeta_f)}/{\Q(\zeta_{\ell})} \right) \bigcup_{i=1}^2 \Gal\left(\Q(\zeta_f)/ H(K_i)\right),
		\end{align*}
		where $\ell$ is an odd prime or $\ell=4 $,  then at least  one of  $K_1,K_2$ has a Euclidean ideal class.
	\end{corollary}
	\begin{proof}
		By 	\Cref{cyclic:classgroup},   $K_1$ and $K_2$  have class number $1$ or $3$.  The result follows from  \Cref{thm:GScubic} and \cite{NW07}.
	\end{proof}

	\begin{corollary}
		Let $K_1$, $K_2$, and $K_3$ be three distinct real quadratic fields with cyclic class groups. Also let   $ H(K_1), H(K_2)$, and $ H(K_3)$, are  abelian extensions over $\Q$. Consider  $f$  to be  the least common multiple  of $f(K_1),f(K_2),f(K_3)$, and $16$.  If 
		\begin{align*}
			\Gal\left(\Q(\zeta_f)/K_1K_2K_3\right)\not \subseteq  \bigcup_ {\ell\mid f}  Gal\left(\Q(\zeta_f)/\Q(\zeta_{\ell}) \right) \bigcup_{i=1}^3 \Gal\left(\Q(\zeta_f)/ H(K_i)\right),
		\end{align*}
		where $\ell$ is an odd prime  or $\ell=4 $, then at least one of  $K_1,K_2,K_3$ has  a Euclidean ideal class.
	\end{corollary}
	\begin{proof}
		By 	\Cref{cyclic:classgroup},   $K_1$ , $K_2$, and $K_3$  have class number $1$ or $2$.  Now the result follows from  \Cref{thm:GSquadratic} and \cite{NW07}.
	\end{proof}

	\section{Consequence of the Elliott-Halberstam Conjecture.}\label{sec:conj}
	
	Let $K$ be an abelian real cubic extension such that  $H(K)$ is also abelian over $\Q$. 
	Let  $f$ be the least common multiple of $16$ and  $f(K) $. And let $G_{\ell}$ be  the Galois group of  $\Q(\zeta_f)$ over $\Q(\zeta_{\ell})$ for all odd prime $\ell$  dividing $f$ or $\ell=4$. 
	
	\begin{theorem}\label{RCE}
		Let $K$ be an abelian real cubic extension with cyclic class group. Equivalently, the class number of $K$ is either $1$ or $3$.
		Also let  $\Gal(H(K)/K)=\langle \sigma\rangle$ and  there exists an extension $\hat{\sigma} $   of  $\sigma $ to $\Q(\zeta_f) $,  such that $\hat{\sigma} \not\in G_{\ell}$ for any odd prime $\ell$ dividing $f$ and $\ell=4$.
		Then  $K$ has a Euclidean ideal class if the Elliott-Halberstam Conjecture is true.
	\end{theorem}
	
	\begin{proof}
		It is clear that $\Gal(H(K)/K)=\langle \sigma\rangle$, and the map   $\hat{\sigma} \in \Gal \left(\Q(\zeta_f)/\Q\right)$, where $\hat{ \sigma}$ is a lift of $\sigma$.  Therefore there exists $a <f$ and coprime to $f$ such that $\sigma_a=\hat{\sigma}.  $
		If $p\equiv a \mod{f}$  and $\fp$ lies over $p$ then the ideal class $[\fp]$  generates the class group $\Cl_K$. The assumption  $\sigma \not \in  G_{\ell}$ 
		implies $(\frac{a-1}{2},l)=1$. 
		By   Lemma \ref{HBL}, with $u=a$  and $v=f$,  we get
		\begin{align*}
			\Bigg| J_{\epsilon}(X), :=&\Big\{p \ \text{is prime}   ~ \Big \vert\ p\equiv a\mod{f} ,\  p\in(X^{1-\epsilon},X) \quad 
			\text{such that} \ \frac{p-1}{2} \ \text{is either } \\  &\text{ prime or product of two primes}\ q_1q_2 \ with \ X^b<q_1<X^c\Big\}\Bigg|\gg \frac{X}{\log ^2X}.
		\end{align*} 
		
	\end{proof}
	For  the cubic number field $K$  in Theorem \ref{RCE}, we define 
	\begin{align*}
		M_{\epsilon}:=&\Big\{\fp\subseteq \cO_K  \ \text{is prime}  ~ \Big \vert\  \Nm(\fp)=p\ \text{ is a prime,} \ p \in J_{\epsilon}(X), \ \text{for some }  X  \Big\}
	\end{align*} 
	and $$M_{\epsilon}(X)=\{\fp\in M_{\epsilon}~ \Big \vert\Nm(\fp)\leq X \}. $$
	
	Since $K$ is a real cubic field, $\cO_K ^{\times}$ contains two multiplicative independent elements, say $ e_1, e_2$.  Assuming  the Elliott-Halberstam Conjecture is true, by   \cite[Lemma 19]{GS20}, there exists a set $P\subseteq M_{\epsilon}$ such that for every $\fp \in P$ for some  $\eta \in \{\pm e_1,\pm e_2\}$  is primitive modulo $\fp$ and there exists an increasing sequence $\{x_n\}_{n\in \N}$ such that
	\begin{align*}
		\left|P(x_n) \right|\gg \frac{x_n}{\log^2 x_n}. 
	\end{align*}
	There are only finitely many ideal classes in $\Cl_K$, we claim that there exists a non-trivial ideal class $[\fa]$ and  $(y_n)_n$, a subsequence of $(x_n)_n$, such that 
	\begin{align*}
		\Bigg| \Big\{ \fp \in [\fa]~\Big \vert ~\fp\in P,\Nm(\fp)\leq y_n \Big\}\Bigg|\gg \frac{y_n}{\log^2y_n}, 
	\end{align*} 
	
	otherwise we get that \[\biggr | P(x_n)\biggr| =o\left(  \frac{x_n}{\log^2x_n}\right).\]
	Since every prime belongs to $M_\epsilon$ lies over some prime $p\equiv a \mod{f}$, therefore for every $\fp\in M_{\epsilon}$,   $[\fp]$ is a generator of the ideal class group. Therefore $[\fa]$ is a generator of $\Cl_K$.  By Theorem \ref{sgl}, $[\fa]$ is a Euclidean ideal class.
	\hfill $\square$
	
	\section{Examples}
	Some examples of   $q,r$, and $s$  such that all primes are congruent to $3$ modulo $4$ and biquadratic fields $\Q(\sqrt{q},\sqrt{rs})$   have class number two,  are presented in the following tables. The reader is referred to \cite{sp3} for a few more examples of biquadratic fields with class number 2 where one of  $r,s$
	is congruent to $1$  modulo $4$.  All examples are computed using SageMath \cite{sage}. The worksheets  can be accessed at the following GitHub repository:  \href{https://github.com/sunilpasupulati/existence-of-a-non-principal-Euclidean-ideal-class-in-biquadratic-fields}{https://github.com/sunilpasupulati/existence-of-a-non-principal-Euclidean-ideal-class-in-biquadratic-fields}.\\

	\begin{center}
		\begin{tabular}{ p{4cm} p{3cm}p{4cm}p{3cm}}
			\hline 	\hline 
			$(q,r,s)$ & $h_{\Q(\sqrt{q},\sqrt{rs})}$ & $(q,r,s)$ & $h_{\Q(\sqrt{q},\sqrt{rs})}$    \\
			\hline	\hline 
			
			(11,19,31) &2  &(11,23,31) &2 \\
			(19,23,31) &2 &(11,47,31) &2 \\
			(47,19,31) &2 &	(47,19,31) &2 \\
			(23,19,31) &2& (59,19,31)&2\\
			\hline
			
		\end{tabular}\\
		
		\vspace{0.4cm}
		Table 1: Examples of $\Q(\sqrt{q},\sqrt{kr})$ with $q\equiv 3, k\equiv 3 , q\equiv 3 \pmod{4}$ and having class number 2.
	\end{center}

	Now we present some examples of biquadratic fields satisfying the assumptions of Corollary \ref{cor32}. These examples are collected from the $L$-functions and modular forms database (LMFDB)~\cite{lmfdb}.
	\begin{center}
		\begin{tabular}{p{2cm} p{2cm}p{3cm}p{2cm}p{2cm}p{3cm}}
			\hline  	\hline  
			$(q,r)$ &$h_{\Q(\sqrt{qr})}$& $h_{\Q(\sqrt{q},\sqrt{r})}$ & $(q,r)$&$h_{\Q(\sqrt{qr})}$ & $h_{\Q(\sqrt{q},\sqrt{r})}$    \\
			\hline
			\hline  
			(17,41)&6 &3 &(5,1637)& 14&7 \\
			(5,157)&6 &3 &(5,1693)& 14&7 \\
			(5,47)& 6&3 &	(2,1009)&14 &7 \\
			(3,49) &6&3 & (2,1013)&14&7\\
			(5,197) &6&3 &	(5,1297) &22&11 \\
			(2,401) &10&5& (2,1297)&22&11\\
			(17,241) &10&5 &	(57,633) &22&11 \\
			(5,577) &10&7& (17,1297)&22&11\\
			\hline
			
		\end{tabular}
		\\
		
		\vspace{0.4cm}
		Table 2: Examples of $\Q(\sqrt{q},\sqrt{r})$ with satisfying assumptions of  Corollary \ref{cor32} 
	\end{center}
	
	The proof of Theorem \ref{thm41} relies on the assumption that $K$ is a totally real field. Without this assumption, the proof would not hold. As a result, the question of whether Euclidean ideals exist in general biquadratic and triquadratic fields remains open.  In case of  totally real triquadratic field $K$, the unit rank is $7$ and therefore by virtue of Graves
	and Murty’s result, it essentially boils down to check whether the Hilbert class field $H(K)$ of $K$ is
	abelian over Q or not. But 
	the existence of Euclidean ideals in general  biquadratic and triquadratic fields is still an unsolved problem. These fields have unique properties, especially when they are not real, which can significantly affect their unit groups so is the ideal class groups. Therefore, investigating these cases requires different techniques and approaches from those used in totally real fields.	
	\section*{Acknowledgments}
	We acknowledge IISER Thiruvananthapuram and IMSc chennai for providing excellent working conditions. 
 The first author thanks to the Institute of Mathematical Sciences (IMSc), where part of this work was conducted, and acknowledges the support received through an IMSc postdoctoral fellowship. We also express our sincere gratitude to Prof. Lawrence C. Washington, Prof. Pieter Moree, and Prof. Sonali Gun for their valuable comments and suggestions. The second author's research was supported by SERB grant CRG/2023/009035.

	\bibliographystyle{amsplain} 
	\bibliography{References.bib}
\end{document}